\documentclass[12pt]{article}
\newtheorem{thm}{Theorem}[section]
\newtheorem{lem}[thm]{Lemma}

\def\rft#1{(\ref{#1}) }
\def\dfn#1{{\it #1}}
\newenvironment{proof}{\noindent{\it Proof} }{\hfill\rule{2mm}{2mm}\vspace{3mm}}

\begin{document}
\centerline{\bf Excluding A Grid Minor In Planar Digraphs}
\vfill
\centerline{Thor Johnson,}
\bigskip
\centerline{Neil Robertson,}
\bigskip
\centerline{Paul Seymour,}
\bigskip
\centerline{and}
\bigskip
\centerline{Robin Thomas}

\vfill
\section*{Abstract}

In~\cite{JohRobSeyThoDirtree}
we introduced the notion of tree-width of directed graphs and presented a conjecture,
formulated during discussions  with Noga Alon and Bruce Reed, stating that a digraph
of huge  tree-width has a large ``cylindrical  grid" minor.
Here  we prove the conjecture for planar digraphs, but many steps of the proof work in general.

This is an unedited and unpolished manuscript  from October 2001.
Since many people asked for copies we are making it available in the hope  that it may be useful.
The conjecture was proved by Kawarabayashi and Kreutzer in {\tt 	arXiv:1411.5681}.

%
\vfill
\pagebreak


\section{Definitions}

All graphs and digraphs in this paper are finite and allowed to have loops and
multiple edges.  We say that an edge of a digraph is directed from its \dfn{tail}
vertex to its \dfn{head} vertex.  An edge $e$ is an \dfn{out-edge} of its tail and
an \dfn{in-edge} of its head.  The \dfn{out-degree} of a vertex is the number of
of out-edges at the vertex and the \dfn{in-degree} is the number of in-edges.  Given
a graph or digraph $G$, $V(G)$ will denote the vertex set and $E(G)$ will denote the
edge set.  Given a digraph $D$ and $X\subseteq V(D)$, $D-X$ denotes the digraph obtained
from $D$ by deleting $X$ and all edges with a head or tail in $X$.  $D\vert X$ denotes
$D-(V(D)-X)$.  We write $\delta ^+(X)$ to mean the set of all edges with head in $X$ and
tail in $V(D)-X$.  $\delta ^-(X)$ denotes $\delta ^+(V(D)-X)$.\\

A \dfn{directed path} in $D$ is a subset of vertices and edges that form a
path in the undirected sense such that no vertex is either the head of two
edges in the path nor the tail.  We say that a directed path \dfn{starts} at
the vertex with in-degree $0$ and \dfn{finishes} at the vertex with out-degree $0$.
The start and finish of a directed path are called the \dfn{terminals} or \dfn{ends}
of the path.  A \dfn{directed circuit} is a directed path with the start vertex equal to
the finish.  A digraph is \dfn{planar} if the underlying graph is planar.\\

We say a digraph $D$ is \dfn{weakly connected} if the underlying graph is
connected.  $D$ is \dfn{strongly connected} if for every $u,v\in V(D)$,
there is a directed path from $u$ to $v$.  A \dfn{strongly connected component}
or \dfn{strong component} of $D$ is a maximal strongly connected subdigraph.\\

We recall that a graph $H$ is a \dfn{minor} of a graph $G$ if $H$ is obtained
from a subgraph of $G$ by contracting edges.  For digraphs, there are many possible
definitions of minor, but in this paper we will be concerned with ``butterfly" minors.
Given a digraph $D$, an edge $e\in E(D)$ is called \dfn{butterfly contractible} if $e$
is the only out-edge of its tail or the only in-edge of its head.  A digraph $H$ is a
\dfn{butterfly minor} (or simply \dfn{minor}) of $D$ if $H$ is obtained from a subdigraph
of $D$ by contracting butterfly contractible edges.\\

A strongly connected digraph $D$ is \dfn{$2$-eulerianizable} if, by adding
parallel edges (same head and tail vertices), we can make every vertex of $D$
have in-degree equal to out-degree and both at most $2$.  We define
\dfn{$4$-eulerianizable} and \dfn{$6$-eulerianizable} in a similar way.\\

For a digraph $D$ and non-negative integer $w$, a \dfn{haven} of order $w$ in $D$
is a function $\cal{B}$ which assigns to every set $Z\subseteq V(D)$ with $\vert
Z\vert <w$ a strong-component ${\cal B}(Z)$ of $D-Z$ such that the following property holds:\\
\begin{itemize}
\item{}If $Z'\subseteq Z\subseteq V(D)$ and $\vert Z\vert<w$, ${\cal B}(Z)\subseteq{\cal B}(Z')$
\end{itemize}
Throughout the paper, we will refer to the above property as the \dfn{haven axiom}.  We note
that if a digraph has a haven of order $w$, it also has a haven of order $w'$ for all
$w'\leq w$.\\

For a positive integer $n$, we construct the \dfn{cylindrical grid} of size $n$ beginning
with a family $C_1$, $C_2$, ..., $C_n$ of $n$ disjoint directed circuits of length $2n$.
We label the vertices on each circuit from $1$ to $2n$ consistent with the circular order
of the directed circuit.  For every $1\leq i\leq n-1$ and every $1\leq j\leq n$, we add an
edge from vertex $j$ on circuit $C_i$ to vertex $j$ on circuit $C_{i+1}$.  Finally, for
every $1\leq i\leq n-1$ and every $n+1\leq j\leq 2n$, we add an edge from vertex $j$ on
circuit $C_{i+1}$ to vertex $j$ on circuit $C_i$.\\


\section{Linkages}

Let $D$ be a digraph and $X\subseteq V(D)$.  We say that $X$ is \dfn{linked} if for
every $A,B\subseteq X$ with $\vert A\vert=\vert B\vert$, there are $\vert A\vert$
vertex-disjoint directed paths in $D$ from the vertices in $A$ to the vertices in $B$.\\

\begin{thm}\label{linked}
Let $n$ be a positive integer and $D$ a digraph with a haven ${\cal B}$ of order $3n$.  There
exists a linked $X\subseteq V(D)$ with $\vert X\vert=2n$.
\end{thm}

\begin{proof}
Let $X\subseteq V(D)$ be a subset of size at most $2n$ with $\vert{\cal B}(X)\vert$ as small
as possible.  Subject to that, we choose $\vert X\vert$ as small as possible.  We claim that
$X$ has size $2n$ and is linked.  For the first part of the claim, suppose that $\vert X\vert
<2n$ and let $y\in{\cal B}(X)$.  Then ${\cal B}(X\cup\{y\})\subseteq{\cal B}(X)$ by the
definition of a haven.  But since $y\in{\cal B}(X)$, we have $\vert{\cal B}(X\cup\{y\})\vert
<\vert{\cal B}(X)\vert$, contradicting the choice of $X$.\\

For the second part of the claim, suppose that $X$ is not linked.  Then there exists disjoint
$A,B\subseteq X$ and $C\subseteq V(D)$ with $\vert A\vert=\vert B\vert >\vert C\vert$ such
that there is no directed path from a vertex in $A$ to a vertex in $B$ in $D-C$.  Because $A$
and $B$ are disjoint subsets of $X$, we must have $\vert X\cup C\vert <3n$.  We set $H={\cal B}
(X\cup C)$.  Since $H$ is strongly-connected and disjoint from $X\cup C$, it will be
contained in a strong component $H'$ of $D-((X-A)\cup C)$.  Suppose it is strictly contained.
We claim that then $H'\cap A\neq\emptyset$.  Otherwise, $H'$ remains strong in $D-(X\cup C)$,
contrary to the fact that $H$ is a strong component in $D-(X\cup C)$.  We conclude that there
is a path from a vertex in $A$ to a vertex in $H$ disjoint from $C$.  A similar analysis
shows that if $H$ is strictly contained in a strong component of $D-((X-B)\cup C)$, there is a
path from a vertex in $H$ to a vertex in $B$ disjoint from $C$.  Since there is no path from
a vertex in $A$ to a vertex in $B$ disjoint from $C$, we conclude that $H$ is a strong
component of either $D-((X-A)\cup C)$ or $D-((X-B)\cup C)$.\\

Without loss of generality, we may assume thet $H$ is a strong component of $D-((X-A)\cup C)$.
Since $(X-A)\cup C\subseteq X\cup C$, the haven axiom shows ${\cal B}(X\cup C)\subseteq
{\cal B}((X-A)\cup C)$.  But since $H={\cal B}(X\cup C)$ and ${\cal B}((X-A)\cup C)$ are both
strong components in $D-((X-A)\cup C$, ${\cal B}(X\cup C)={\cal B}((X-A)\cup C)$.  Applying
the haven axiom to $X$ and $X\cup C$ then shows ${\cal B}((X-A)\cup C)\subseteq
{\cal B}(X)$.  Since $\vert C\vert <\vert A\vert$, $\vert (X-A)\cup C\vert <\vert X\vert=2n$.
These last two conclusions regarding $(X-A)\cup C$ contradict the choice of $X$. (We see
that either $X$ can be chosen to make ${\cal B}(X)$ smaller or, if not that, $X$ itself can
be made smaller keeping ${\cal B}(X)$ the same.)
\end{proof}

Given a linked set $X$ and two disjoint subsets $A$, $B$ of the same size, there are
vertex-disjoint directed paths linking the vertices in $A$ to the vertices in $B$.  We call
this collection of paths a \dfn{linkage} from A to B.  If a path in the linkage begins at a
vertex $a\in A$ and ends at a vertex $b\in B$, we say that $a$ is linked to $b$ in the
linkage and $b$ is linked from $a$ in the linkage.\\

If linear orders are defined on $A$ and $B$, the linkage is \dfn{monotone
increasing} if for every $a_1$,$a_2\in A$ with $a_1<a_2$, the vertex in $B$ linked from
$a_1$ occurs before the vertex in $B$ linked from $a_2$.  The linkage is \dfn{monotone
decreasing} if the linkage is monotone increasing when the order on $B$ is reversed.  The
linkage is \dfn{monotone} if it is either monotone increasing or monotone decreasing.\\
Because $X$ is linked, there is also a linkage from the vertices in $B$ to those in $A$.
If there exist orders on $A$ and $B$ such that the two linkages are both monotone increasing,
the pair of linkages is said to \dfn{agree}.  If there exist orders on $A$ and $B$ such that
one linkage is monotone increasing and the other is monotone decreasing, the pair of linkages
is said to \dfn{cross}.

\begin{lem}\label{linkmatch}
Let $n$ be a positive integer.  There exists a positive integer $N=N(n)$ such that the
following holds: If $X$ is a linked set in a digraph $D$ and $A$,$B\subseteq X$ are disjoint
with $\vert A\vert=\vert B\vert=N$, there are $A'\subseteq A$ and $B'\subseteq B$ with
$\vert A'\vert =\vert B'\vert=n$.  In addition, there are $n$ vertex-disjoint directed
paths linking the vertices in $A'$ to those in $B'$, $n$ vertex-disjoint directed paths
linking the vertices in $B'$ to those in $A'$, and these two linkages either agree or
cross.
\end{lem}

\begin{proof}
We begin by fixing $A_0\subseteq A$ of size $n_0=n^2+1$.  We next arbitrarily order the
elements of $B$.  For every $B_0\subseteq B$ of size $n_0$, there is a linkage from $A_0$
to $B_0$.  Using the matching defined by this linkage, the order on $B_0$ (inherited
from $B$) defines an order on $A_0$.  Applying Ramsey theory, for $N$ sufficiently
large compared to $n_0$ we may select $B_1\subseteq B$ with $\vert B_1\vert=n_0^2$ such
that every subset of $B_1$ of size $n_0$ defines the same order on $A_0$.  We fix this
order on $A_0$ and number the vertices in $A_0$ from $1$ to $n_0$ to agree with this
ordering.\\

We number the vertices of $B_1$ from $1$ to $n_0^2$ to agree with the ordering on $B_1$.  We
form the subset $B_0$ consisting of those vertices numbered $i(n_0-1)+1$ $\forall 1\leq
i\leq n_0$.  Since $X$ is linked, there is a linkage $L$ from $B_0$ to $A_0$.  We apply
Erods-Sekeres to find $A'\subseteq A_0$ and $B'\subseteq B_0$ with $\vert A'\vert=\vert
B'\vert=n$ such that the vertices in $B'$ are matched to those in $A'$ under $L$ and this
restriction of the linkage is monotone.\\

The vertices in $A_0$ are numbered, and we let the numbers on the vertices in $A'$ be
$\{i_1,i_2,...,i_n\}$ with $i_j<i_k$ when $j<k$.  Similarly we let the vertices in $B'$ be
$\{b_1,b_2,...,b_n\}$ with $b_i$ occuring before $b_j$ in the order on $B_1$ when $i\leq j$.
We form $B''$ with $B'\subset B''\subset B_1$ of size $n_0$ by starting with $B'$.  We then
select $i_1-1$ elements of $B_1$ which occur before $b_1$ and $n_0-i_n$ elements which occur
after $b_n$.  Finally, for $1\leq j<n$, we select $i_{j+1}-i_j$ elements which occur between
$b_{j+1}$ and $b_j$. (These choices are all possible since $B'$ is a subset of $B_0$.) Since
$B''$ is a subset of $B_1$ of size $n_0$, there is a linkage from $A_0$ to $B''$ that is
monotone increasing.  This linkage restricts to a monotone increasing linkage between $A'$
and $B'$.  Combined with the linkage found in the previous paragraph, this completes the
proof.
\end{proof}

\begin{thm}\label{linkmatch2} Let $n$ be a positive integer.  There exists a positive
integer $N=N(n)$ such that if $X$ is a linked set in a digraph $D$ and $A,B\subseteq X$
are disjoint with $\vert A\vert=\vert B\vert=N$, one of the following holds:\\
\begin{enumerate}
\item{}There are $A^*\subseteq A$ and $B^*\subseteq B$ with $\vert A^*\vert =\vert B^*\vert=n$.
In addition, there are $n$ vertex-disjoint directed paths linking the vertices in $A^*$ to
those in $B^*$, $n$ vertex-disjoint paths linking the vertices in $B^*$ to those in $A^*$,
and these two linkages agree.  Finally, the two linkages are disjoint expect that the
directed paths between the same two vertices may share vertices and edges.
\item{}There is a $4$-eulerianizable subdigraph $D'$ of $D$ with a haven of order $n+1$.
\end{enumerate}
\end{thm}

\begin{proof}
We may assume that $n$ is even.  We apply \rft{linkmatch} with the $n$ in that
lemma equal to $n'$, $n'$ to be chosen as the proof proceeds but at least $n$.  Let $A'$
and $B'$ be as in the outcome of \rft{linkmatch}.\\

We next form a collection of subdigraphs.  In the case that the linkages from \rft{linkmatch}
agree, the $n'$ subdigraphs are formed by taking the union of a path in one linkage with a
path in the other with the same end-vertices.  In the case that the linkages cross, $\frac{n'}
{2}$ subdigraphs are formed by taking the union of two paths in one linkage and two paths in
the other chosen so that each has its end-vertices among a fixed set of size $4$.\\

By choosing $n'$ sufficiently large, there are either $n$ of these subdigraphs that are
vertex-disjoint or there are $2n+1$ that pairwise share a vertex.  In the first case, we get
the first outcome by selecting a vertex in $A$ and a vertex in $B$ from each subdigraph.  In
the second case, let $D'$ be the restriction of $D$ to the $2n+1$ subdigraphs that pairwise
share a vertex.  Every vertex is contained in at most one of the paths in each linkage, so
every vertex is contained in at most two of the subdigraphs.  By adding parallel edges to
$D'$, we may make the subdigraphs edge-disjoint and thereby make $D'$ eulerian.  Since the
maximum out-degree of a veretx in any of the sub-digraphs is $2$, we see that $D'$ is
$4$-eulerianizable.\\

To show that $D$ has a haven of order $n+1$, we again use the fact that every vertex is in
at most $2$ subdigraphs.  Therefore, for every set $X\subseteq V(D')$ with $\vert X\vert\leq
n$, there is at least one subdigraph disjoint from $X$.  Since every two subdigraphs have a
vertex in common, we construct the necessary haven by letting ${\cal B}(X)$ be the strong
component of $D'-X$ containing all of the subdigraphs disjoint from $X$.  
\end{proof}


\section{Reducing To Eulerian}

Our goal in this paper is to prove a digraph which has a haven of large order contains a
large cylindrical grid minor.  In this section, we show that a general digraph with a large
haven contains a subdigraph which is $6$-eulerianizable which also has a large haven.  We
begin with a definition.  Let $G$ be a graph and $D$ a digraph.  Suppose that $D$ has two
collections ${\cal V}$,${\cal E}$ of vertex-disjoint strongly connected subdigraphs with the
subdigraphs in ${\cal E}$ $2$-eulerianizable.  Suppose further that there are bijections
$\nu:{\cal V}\rightarrow V(G)$ and $\epsilon:{\cal E}\rightarrow E(G)$ such that if $e\in{\cal E}$ is
mapped to an edge with end-vertices $u$ and $v$ in $G$, $e$ shares a vertex with each of the
two elements in ${\cal V}$ mapped to $\{u,v\}$.  We then say that $(D,{\cal V},{\cal E},\nu,
\epsilon)$ or $(D,{\cal V},{\cal E})$ is a \dfn{representation} of $G$ or that $D$ has a
representation of $G$.  If every $e\in{\cal E}$ is disjoint from all $v\in{\cal V}$ except
those $v$ with $\nu(v)$ an end of $\epsilon(v)$, we say that the representation is
\dfn{faithful}.\\

\begin{thm}\label{representationhaven}
Let $D$ be a digraph which has a faithful representation $(D,{\cal V},{\cal E},\nu,\epsilon)$
of some graph $G$.  Suppose that $G$ has a haven of order $h+1$, $h$ a positive even integer.
Then $D$ has a haven of order $h/2+1$.
\end{thm}

\begin{proof}
We begin by constructing a map $m$ from $V(D)$ to $2^{V(G)}$.  If $v\in V(D)$ is disjoint
from every member of ${\cal V}$ and ${\cal E}$, we set $m(v)=\emptyset$.  If $v$ is disjoint
from every member of ${\cal E}$ but intersects $u\in{\cal V}$, we set $m(v)=\{\nu(u)\}$.
Otherwise, $v$ intersects some $e\in{\cal E}$ and we set $m(v)=\{u,u'\}$ where $u$ and $u'$
are the ends of $\epsilon(e)$.\\

Let ${\cal B}_G$ be a haven of order $h+1$ for $G$.  We will construct a haven ${\cal B}_D$
of order $h/2+1$ for $D$.  Let $X\in V(D)$ with $\vert X\vert\leq n/2$.  Let $Y=\cup_{x\in
X}m(x)$.  Then $\vert Y\vert\leq h$ so ${\cal B}_G(Y)$ exists.  Let $u\in{\cal B}_G(Y)$.
We set ${\cal B}_D(X)$ to be the strong component that contains $\nu^{-1}(u)$.  It is easy
to check that this is uniquely defined and gives a haven in $D$ of order $h/2+1$.
\end{proof}

\begin{thm}\label{lemma1}
Let $n$ be a positive integer.  There exists a positive integer $N=N(n)$ such that
if $D$ is a digraph containing a linked set of size at least $N$, one of the following
holds.\\
\begin{itemize}
\item{}$D$ has a $4$-eulerianizable sub-digraph with a haven of order at least $n$.
\item{}There is a representation of $K_n$ in $D$.
\end{itemize}
\end{thm}

\begin{proof}
The proof will consist of proving two claims.\\

\begin{it}
\noindent Claim 1: Let $n$ be a positive integer.  There exists a positive integer
$N=N(n)$ such that the following holds: Let ${\cal V}$ be a family of vertex-disjoint
strongly connected subdigraphs of a digraph $D$.  Let $\vert{\cal V}\vert =N$ and let $D$
have a linked set $X$ such that each member of ${\cal V}$ contains at
least some number $c$ elements of $X$.  Then, one of the following holds.\\
\begin{enumerate}
\item{}$D$ has a $4$-eulerianizable sub-digraph with a haven of order at least $n+1$.
\item{}There is a family ${\cal V'}$, $\vert{\cal V'}\vert =n$, of vertex-disjoint strongly
connected subdigraphs and a second family ${\cal E}$, $\vert{\cal E}\vert =n(n-1)/2$, of
vertex-disjoint, $2$-eulerianizable subdigraphs.  For every pair of members of ${\cal V'}$,
there is a distinct member of ${\cal E}$ that meets both in at least one vertex.
\item{}There is a family ${\cal V'}$, $\vert{\cal V'}\vert =n$, of vertex-disjoint strongly
connected subdigraphs such that each member of ${\cal V'}$ contains at least $c+1$
elements of $X$.
\end{enumerate}
\end{it}

\noindent {\it Proof of claim:} By ignoring a large number of members of ${\cal V}$, we may
assume that there are at least $\vert{\cal V}\vert$ elements of $X$ that are not in any member
of ${\cal V}$ while keeping $\vert{\cal V}\vert$ arbitrarily large.  We now select one member
of $X$ from each member of ${\cal V}$ to form the set $A$, and we select $\vert{\cal V}\vert$
elements of $X$ disjoint from ${\cal V}$ to form the set $B$.  We apply \rft{linkmatch2} with
the $n$ of that lemma equal to some positive integer $m$ to be determined.  For sufficiently
large $\vert{\cal V}\vert$, one result of \rft{linkmatch2} is the first outcome of this claim,
in which case the proof of the claim is complete.  The other outcome is two subsets $A'
\subseteq A$ and $B'\subseteq B$ of size $m$, a linkage from $A'$ to $B'$, and a linkage
from $B'$ to $A'$.  Moreover, these two linkages agree and are pairwise vertex-disjoint
except for those paths with the same end-vertices.  We let ${\cal V}_1$ be the $m$ members
of ${\cal V}$ containing an element of $A'$.  For every $v\in{\cal V}_1$, we call the two paths
linking its element in $A'$ to and from an element in $B'$ the \dfn{linking} associated with
$v$.\\

We form a digraph $G$ with vertex set the members of ${\cal V}_1$, and we add an edge from
$v$ to $u$ if the linking associated with $v$ shares a vertex with $u$.  For sufficiently
large $m$, $G$ has either an independent set ${\cal I}$ of size $n$ or a tournament subdigraph
of order $m'$ ($m'$ to be chosen).  In the first case, we augment each member of ${\cal I}$
by its associated linking.  The result is that each member of ${\cal I}$ remains strongly
connected and now contains at least $c+1$ elements of $X$.\\

For the second case, we recall that for any given $k$, there is a $K$ such that every
tournament on $K$ vertices contains a transitive tournament on $k$ vertices.  So we may
choose $m'$ such that there is a transitive tournament subdigraph of $G$ with $n*(n+1)/2$
vertices.  We let ${\cal V}_2$ be the vertex set of this tournament and order the elements
of ${\cal V}_2$ from highest out-degree in the transitive tournament to lowest.  Let the
family of linkings associated with the first $n$ members of ${\cal V}_2$ be ${\cal V'}$.
The members of ${\cal V}'$ are strongly connected and pairwise vertex-disjoint.  Moreover,
each of the $\frac{n(n-1)}{2}$ last members of ${\cal V}_2$ meets every member of ${\cal V'}$.
So, for any pair of members of ${\cal V'}$, each of the $\frac{n(n-1)}{2}$ last members of
${\cal V}_2$ contains a $2$-eulerianizable subdigraph meeting both in at least one vertex.  We form a
bijection between the last $\frac{n(n-1)}{2}$ members of ${\cal V_2}$ and the pairs of elements of
${\cal V'}$.  We select a $2$-eulerianizable subdigraph from each of the $\frac{n(n-1)}{2}$ last
members of ${\cal V_2}$ containing a vertex from each member of its associated pair.  These
euleriaizable subdigraphs form ${\cal E}$ in the second outcome of the claim.  This completes
the proof of the claim.\\

The proof of the theorem will follow from the proof of the following claim by taking $c=1$
and taking the family ${\cal V}$ to be singletons.\\

\begin{it}
\noindent Claim 2: Let $n$ be a positive integer.  There exists a positive integer
$N=N(n)$ such that the following holds: Let ${\cal V}$ be a family of vertex-disjoint
strongly connected subdigraphs of a digraph $D$.  Let $\vert{\cal V}\vert =N$ and let $D$
have a linked set $X$ such that each member of ${\cal V}$ contains at least some number $c$
elements of $X$.  Then one of the following holds:\\
\begin{enumerate}
\item{}$D$ has a $4$-eulerianizable sub-digraph with a haven of order at least $n+1$.
\item{}There is a family ${\cal V'}$, $\vert{\cal V'}\vert =n$, of vertex-disjoint strongly
connected subdigraphs and a second family ${\cal E}$, $\vert{\cal E}\vert =n(n-1)/2$, of
vertex-disjoint, $2$-eulerianizable subdigraphs.  For every pair of members of ${\cal V'}$,
there is a distinct member of ${\cal E}$ that meets both in at least one vertex.
\end{enumerate}
\end{it}

\noindent {\it Proof of claim:} We first prove the claim in the case that $c$ is
sufficiently large.  In this case, we restrict ${\cal V}$ to a subset ${\cal V}_1$ of size $n$.
For every pair of members $u$,$v\in{\cal V}_1$, we apply \rft{linkmatch2} with $A$ equal
to $c$ of the vertices in $X$ that are in $u$; $B$ equal to $c$ of the vertices in $X$ that
are in $v$; and $n$ of that lemma equal to $m'$, $m'$ to be determined but larger than $n$.
We may assume that we always get outcome $1$ of \rft{linkmatch2} as otherwise we are done
with the proof of the claim (since $m'\geq n$).  We call the two linkages of outcome $1$ the
\dfn{linking} associated with the pair $(u,v)$.  The size of the linking is $\vert A\vert$.
The union of a path from a vertex in $A$ to a vertex in $B$ with the path from the same vertex
in $B$ back to the vertex in $A$ is called a \dfn{connector}.  By a \dfn{partial linking} we
mean a subset of the linkage from $A$ to $B$ and a subset of the linkage from $B$ to $A$ that
has the same end-vertices as the first linkage. (In other words, a subset of the connectors.)
By bipartite Ramsey, for any positive integer $k$ there is a positive integer $K$ such that
the following holds: Given a linking associated with a pair $(u_1,v_1)$ and a linking
associated with a pair $(u_2,v_2)$ both of size $K$, there are partial linkings of size $k$
for the pairs $(u_1,v_1)$ and $(u_2,v_2)$, subsets of the original linkings.  Moreover, either
every connector of the first linking is disjoint from every connector of the second or every
connector of the first shares a vertex with every connector of the second.  By applying this
to every pair of pairs, we find, for sufficiently large $m'$, either a collection of linkings
of size one, one for each pair, that are pairwise disjoint, or a linking ${\cal L}_1$ of
size $n$ and a linking ${\cal L}_2$ of size $\frac{n(n-1)}{2}$ with pairwise intersecting connectors.
In the first case, we let ${\cal V'}={\cal V}_1$ and let ${\cal E}$ be the set of linkings to
get the second outcome of the claim.  In the second case, we let ${\cal V'}$ be the collection
of connectors in ${\cal L}_1$ and ${\cal E}$ be the collection of connectors in ${\cal L}_2$ to
again produce the second outcome.  This proves the claim in the case of sufficiently large $c$.\\

So, suppose the claim is false in general and let $c_0$ be the largest $c$ which is a
counter-example.  We then apply the first claim with $n$ of claim $1$ equal to the
required $N$ of claim $2$ when $c>c_0$.  Claim $1$ must result in outcome (3), which is the
hypothesis of claim $2$ with $c=c_0+1$.  Since claim $2$ is true in this case, we have a
contradiction.  So claim $2$ is true, and the proof of the theorem is complete.
\end{proof}

Before proceeding we need the following two technical lemmas.
               
\begin{lem}\label{markededges} Let $n$ and $b$ be non-negative integers.  There exists
a positive integer $N=N(n,b)$ such that the following holds: We label each edge of $K_N$ with
a subset of size at most $b$ of the vertex set such that no edge receives a label that is
equal to one of its end-vertices.  Then there is a $K_n$ subgraph such that no edge in the
clique has a label of any vertex in the clique.
\end{lem}

\begin{proof}
Suppose the result is false.  We choose a counter-example with $b+n$ minimum.  Clearly $b$,
$n\neq 0$.  By our choice of counter-example, the theorem is true for $b-1$,$n$ and $b$,$n-1$.
Let $v\in K_N$, $N$ to be chosen.  If $N$ is sufficiently large, there is either a
$K_{N(n,b-1)}$ ($N(n,b-1)$ exists by our choice of $n$,$b$) subgraph with edges which have
the label $v$ or a $K_m$, $m$ to be determined, subgraph with edges that don't have label $v$.
In the first case we restrict attention to the $K_{N(n,b-1)}$ subgraph.  Since $v$ is not a
vertex of the $K_{N(n,b-1)}$ subgraph, we may regard the edges in this subgraph as having at
most $b-1$ labels.  We then find the desired $K_n$.\\

So we may suppose that we have a $K_m$ subgraph with edges that do not have the label $v$.  We
say that $u$ \dfn{prevents} $u'$ if the edge $vu$ has the label $u'$.  We construct a digraph
$D$ on the $m$ vertices of the $K_m$ subgraph by adding an edge from $u$ to $u'$ if $u$
prevents $u'$.  The maximum out-degree of this digraph is $b$.  We claim that for $m$
sufficiently large, there is an independent set $I$ of size $N(n-1,b)$ of the underlying
undirected graph of $D$.  Otherwise there is a $K_{2b+1}$ subgraph of the underlying
undirected graph, and this contradicts the bound on the maximum out-degree (since some
vertex of the clique has at least as many out-edges as in-edges within the clique).  Back in
the labeled graph, since $I$ has size $N(n-1,b)$, we may apply induction to find a $K_{n-1}$
subgraph such that no edge in the clique has a label of any vertex in the clique.  We claim we
can add $v$ to this clique to produce the desired $K_n$.  $v$ does not appear as a label on
any edge because the $K_{n-1}$ subgraph is contained in the $K_m$ subgraph constructed
earlier.  For any other vertex $u'$ to appear as a label on an edge in the $K_n$, it must
appear as a label on one of the edges incident to $v$, say $vu$.  But this means there is
an edge from $u$ to $u'$ in $D$, which contradicts the fact that $u$ and $u'$ belong to the
same independent set in $D$.
\end{proof}

We will use the previous result to prove the following more general lemma.\\

\begin{lem}\label{generalmarkededges} Let $m$ and $n$ be positive integers.  There exists
a positive integer $N=N(m_1,m_2,n)$ such that the following holds: We label each edge of
$K_N$ with a subset of the vertex set such that no edge receives a label that is equal to
one of its end-vertices.  Then one of the following occurs:\\
\begin{enumerate}
\item{}There is $K_n$ subgraph such that no edge in the clique has a label of any vertex
in the clique.
\item{}There are $m_1$ edges and $m_2$ vertices such that each of the vertices appears as
a label on each of the edges.
\end{enumerate}
\end{lem}

\begin{proof}
Suppose the theorem is false and we choose a counterexample with $m$ as small as possible.
By \rft{markededges}, $m_1>1$.  By choosing $N$ sufficiently lage, \rft{markededges} shows
that any labeled $K_N$ for which the theorem fails must have an edge $e$ receiving at least
$N(m_1-1,m_2,n)$ labels.  Let $K$ be the clique induced on the vertices appearing as labels
on $e$.  This clique either contains a $K_n$ subgraph that satisfies the first outcome or
has $m_1-1$ edges and $m_2$ vertices such that each of the vertices appears as a label on
each of the edges.  We add $e$ to the collection of edges to get the second outcome.
\end{proof}

We are now ready to present the last main lemma needed for reducing the general case.  It
is concerned with the case when \rft{lemma1} results in outcome $2$.\\

\begin{lem}\label{lemma2} Let $h$ be a positive integer and $N$ be as in
\rft{generalmarkededges}.  Let $k=N(8h^2,N(8h^2,h,h),N(8h^2,h,h)(N(8h^2,h,h)-1)/2)$.  If
$D$ is a digraph with a representation of $K_k$, $D$ has a $5$-eulerianizable subdigraph with
a haven of order at least $h+1$.
\end{lem}

\begin{proof}
Let $(D,{\cal V},{\cal E},\nu,\epsilon)$ be a representation of $K_k$.  We will label the edges
of the clique, adding a label $u\in V(K_N)$ to an edge $e\in E(K_N)$ if $u$ is not an end of
$e$ and $\epsilon ^{-1}(e)$ shares a vertex with $\nu ^{-1}(u)$.  We then apply
\rft{generalmarkededges} with $n=8h^2$, $m_1=N(8h^2,h,h)$, and $m_2=N(8h^2,h,h)(N(8h^2,h,h)-1)
/2$.  We first assume that outcome $1$ occurs.  Let $G$ be the cubic grid on $2h$ by $4h$
vertices.  We let ${\cal V'}$ be the sub-family of ${\cal V}$ corresponding to the vertices
used in the clique we have found.  We restrict ${\cal E'}$ so that $(D,{\cal V'},{\cal E'})$ is a
representation of $G$.  Since the maximum degree of a vertex in a cubic grid is $3$, we may
find a $3$-eulerianizable subdigraph of each $v\in{\cal V'}$ which meets the same members of
${\cal E'}$ as $v$.  The union $D'$ of the members of ${\cal E'}$ and these $3$-eulerianizable
subdigraphs is $5$-eulerianizable and has a faithful representation of $G$. By
\rft{representationhaven}, $D'$ has a haven of order $h+1$ as required.\\

We next consider the case when outcome $2$ of \rft{generalmarkededges} occurs.  Let
${\cal V'}\subseteq{\cal V}$ and ${\cal E'}\subseteq{\cal E}$ correspond to the vertices and
edges of outcome $2$.  Now, every member of ${\cal V'}$ shares a vertex with every member of
${\cal E'}$ and $\vert{\cal V'}\vert=\vert{\cal E'}\vert(\vert{\cal E'}\vert -1)/2$.  So, we may
create a bijection between the elements of ${\cal V'}$ and the pairs of elements of ${\cal E'}$
such that the element assigned to a pair intersects a vertex from each member of that pair.
Moreover, each member of ${\cal V'}$ can be restricted to a $2$-eulerianizable strongly
connected digraph (a subdigraph of the original element) which still meets both elements of
the assigned pair in ${\cal E'}$.  The collection ${\cal V''}$ of these $2$-eulerianizable
subdigraps makes $(D,{\cal E'},{\cal V''})$ a representation of $K_{N(t+1,t,t)}$. (Note that
${\cal E'}$ corresponds to the vertices and ${\cal V''}$ corresponds to the edges.)  We proceed
as in the beginning of this proof, forming the labeled clique and applying
\rft{generalmarkededges} with $n=8h^2$, $m_1=m_2=h+1$.  In case of outcome $1$ we proceed
as in the previous paragraph.  In case of outcome $2$, let ${\cal E'}_0\subseteq{\cal E'}$ and
${\cal V''}_0\subseteq{\cal V''}$ be the subdigraphs corresponding to the vertices and edges
in outcome $2$.  The union of these collections is $4$-eulerianizable and every member of
${\cal V''}_0$ shares a vertex with every member of ${\cal E'}_0$.  Moreover, any $h$ vertices
in $D$ are disjoint from a member of ${\cal V''}_0$ and a member of ${\cal E'}_0$.  It is not
difficult to see this means the union has a haven of order $h+1$.
\end{proof}


\section{Waste Not, Want Not}

In the remainder of the paper, we prove that planar digraphs with large tree-width have large
cylindrical grid minors.  Let $\Sigma$ be a closed disk and let $D$ be a digraph.
Let $T$, $B$, $L$, and $R$ be pairwise disjoint subsets of $V(D)$ with $\vert T\vert =\vert
B\vert$ and $\vert L\vert=\vert R\vert$.  Let the vertices in $B$ and $R$ have out-degree
$1$ and in-degree $0$ and those in $T$ and $L$ have in-degree $1$ and out-degree $0$.  Suppose
that $D$ is drawn in $\Sigma$ with no edge crossings and such that the vertices drawn on
the boundary of the disk are exactly those in $T$, $B$, $L$, and $R$.  Furthermore, in the
counter-clockwise orientation of the disk, suppose all the vertices in $T$ are followed by
those in $L$, $B$, and $R$ respectively.  We say that $D$ is \dfn{$\{T,L,B,R\}$ disk-embedded}
in $\Sigma$.  Given such an embedding, if there is a simple curve with both ends on the
boundary of $\Sigma$ which separates $L$ and $R$, the side that contains $R$ will be called
the \dfn{right} side and the other the \dfn{left} side.  We may then define a left-right
order on a set of disjoint curves of this kind.  We define \dfn{top}, \dfn{bottom}, and a
top-bottom order in a similar way.\\

Let $P$ be a directed path containing vertices $u$,$v$ such that there is a sub-path $P'$
from $u$ to $v$.  We denote $P'$ by $P[u,v]$.  Given two directed paths $P_1$, $P_2$ of a
digraph, we say that they hit in \dfn{reverse} if there do not exist vertices $v_1$, $v_2$
such that $v_1$ appears later than $v_2$ in both $P_1$ and $P_2$.  Consider a $\{T,L,B,R\}$
disk-embedded digraph $D$ and two edge-disjoint paths $P$ and $Q$ from a vertex in $R$ to
a vertex in $L$ and from a vertex in $B$ to a vertex in $T$ respectively.  We call $P$ a
\dfn{horizontal} path and $Q$ a \dfn{vertical} path.  Since $G$ is embedded, $P$ and $Q$
must share a vertex.  Let $u$ be such a vertex.  If the in-edge and out-edge of $u$ in $Q$
are both on the top or both on the bottom of $P$, we say that $Q$ \dfn{bounces} at $u$.
Otherwise, we say $Q$ \dfn{crosses} at $u$.  Note that if $e$ and $f$ are edges of $v$ with
the head of $e$ equal to the tail of $f$, $e$ and $f$ are on the same side of $P$ (top or
bottom) iff $p$ bounces at the head of $e$.\\

\begin{thm}\label{minunion} Let $D$ be $\{T,L,B,R\}$ embedded in a closed disk $\Sigma$.
Suppose that there are $\vert T\vert$ vertex-disjoint directed paths from $B$ to $T$ and
$\vert R\vert$ vertex-disjoint directed paths from $R$ to $L$.  There is a butterfly minor
$D'$ of $D$ with the following properties:\\
\begin{enumerate}
\item{}$D'$ consists of $\vert T\vert$ vertex-disjoint directed paths from $B$ to $T$ and
$\vert R\vert$ vertex-disjoint directed paths from $R$ to $L$.
\item{}The paths in (i) hit in reverse.
\item{}All vertical paths are to the left of the right-most original vertical path.  All
horizontal paths are below the top-most original horizontal path.
\end{enumerate}
\end{thm}

\begin{proof}
We consider a routing of the two sets of paths using the fewest number of edges subject to
property $3$ and delete all unused vertices and edges.  We butterfly contract any edge
appearing in two paths (necessarily a vertical and horizontal path).  We will prove by
induction on $k$ that the $k$ bottom-most horizontal paths hit all the vertical paths in
reverse.  The base case of $k=0$ is trivial.\\

Suppose for the sake of a contradiction that the $k$ bottom-most horizontal paths hit all
the vertical paths in reverse but the $(k+1)^{st}$ horizontal path does not.  We label this
path $P$.  Let $Q$ be some vertical path which $P$ does not hit in reverse.  There is a
sub-path $P'$ of $P$ that intersects $Q$ precisely at its end-vertices but does not hit
$Q$ in reverse.  $P'$ lies either to the left or right of $Q$.  We may assume that $Q$ and
$P'$ are chosen so that $P'$ lies to the right of $Q$ if possible and, subject to that, that
$P'$ is minimal (allowing the possibility of changing vertical paths).  Let $P'$ start at the
vertex $v_1$ and end at $v_2$.\\

Suppose first that $P'$ lies to the right of $Q$. We claim that none of the $k$ bottom-most
horizontal paths can have a vertex on $Q$ between $v_1$ and $v_2$.  Suppose that this is
false and let $P_1$ be one of the $k$ bottom-most horizontal paths which does have such
a vertex.  We next suppose that $P_1$ crosses $Q$ between $v_1$ and $v_2$.  Let $u$ be the
first vertex on $P_1$ which is a vertex where $P_1$ crosses $Q$ between $v_1$ and $v_2$.
Since $P$ and $P_1$ are vertex-disjoint and $P'$ lies to the right of $Q$, $P_1$ must cross
from the left of $Q$ to the right at $u$.  So, the edge on $Q$ with tail is $u$ is on the
bottom of $P_1$.  But by induction and our choice of $u$, $P_1$ does not cross $Q$ between
$u$ and $v_2$.  So, $v_2$ is on the bottom of $P_1$, which contradicts the fact that $P$ is
on the top of $P_1$.  So, $P_1$ must not cross $Q$ between $v_1$ and $v_2$.  Now let $u$ be
any vertex of $P_1$ that is on $Q$ between $v_1$ and $v_2$.  We have just shown that $P_1$
must bounce at $u$, and it is not difficult to see (since $P_1$ doesn't cross $Q$ between
$v_1$ and $v_2$) that $P_1$ must be to the left of $Q$ before (and after) this bounce.  If
the edge on $Q$ with tail $u$ is above $P_1$, $P_1$ and $Q$ can not hit in reverse, a
contradiction to the inductive hypothesis.  So, the edge on $Q$ with tail $u$ is below
$P_1$.  Since $P_1$ doesn't cross $Q$ between $v_1$ and $v_2$, we again reach the
contradiction that $v_2$ is on the bottom of $P_1$.  Therefore, none of the $k$ top-most
horizontal paths can have a vertex on $Q$ between $v_1$ and $v_2$.\\

If there is a vertex $u\in Q[v_1,v_2]$ such that $P$ crosses $Q$ at $u$, let $u_0$ be the
earliest such vertex on $P$.  If $u_0$ exists and occurs on $P$ before $v_1$, let $v=u_0$.
Otherwise, let $v=u$.  We claim that the edge $e$ on $Q$ with tail $v$ is on bottom of $P$.
In the case that $v\neq v_1$, the claim follows because the first crossing must be from the
left to the right.  When $v=v_1$, we may assume that $P$ bounces at $v_1$ as otherwise the
claim holds for the same reason as when $v\neq v_1$.  Now the claim holds as otherwise $P$
must bounce at $v_1$ in such a way that $u_0$ must exist and occur on $P$ before $v_1$.\\

Starting from $e$ we follow $Q$ until we reach either $v_2$ or a vertex where $P$ crosses
$Q$.  As we are always staying on the bottom of $P$, we will not encounter any vertices on
paths above $P$, and we have already shown that we will not encounter any vertices on paths
below.  If we reach $v_2$, the sub-path of $Q$ we have followed gives a way to re-route $P$
(since $u$ comes before $v_2$) without destroying the other horizontal paths while maintaining
property $3$.  Moreover, the edge on $P$ with tail $v$ is no longer needed, a contradiction
to minimality.  If we reach a crossing before $v_2$, the crossing vertex occurs later on $P$
than $v$ by the choice of $u$, and so again we may re-route.\\

This concludes the case when the sub-path of $P$ between $v_1$ and $v_2$ lies to the right
of $Q$.  The remaining case is when it lies to the left.  If this sub-path does not contain
any vertices in its interior on any vertical path, we may re-route $Q$ along this sub-path
while maintaining property $3$, contradicting minimality.  Let $Q_1$ be the vertical path
immediatly to the left of $Q$.  Let $v_3$ be the first vertex on $P$ after $v_1$ that is on
$Q_1$ and let $v_4$ be the last vertex on $P$ before $v_2$ that is on $Q_2$. (These exist by
what we have just said.) If $v_3$ and $v_4$ are distinct, $v_3$ must occur on $Q_1$ before
$v_4$.  We then have a contradiction to our choice of $P'$.  So, let $u=v_3=v_4$.  Now,
there must be a vertex on $P$ after $v_2$ that is also on $Q_1$.  Let the first such vertex
be $x$.  By our choice of $P'$, $x$ must occur on $Q_1$ before $u$.  Consider the directed
circuit obtained by concatenating $P[u,x]$ with $Q_1[x,u]$. (There are no repeated vertices
due to the choice of $x$.) Since $P$ does not intersect itself and begins in $R$, there is
a vertex $y$ on $P$ before $u$ that is on $Q_1$ before $u$.  But since $v$ appears on $P$
before $u$, $y$ also appears on $P$ before $u$.  So $P[y,u]$ shows that $P$ and $Q_1$ do not
hit in reverse.  Moreover, this sub-path lies to the right of $Q_1$, contradicting our choice
of $P'$.
\end{proof}

We will need a version of the above theorem for the case when $D$ is embedded in a cylinder
(disk with a hole removed) rather than a disk. Let $\Sigma$ be a closed cylinder and let $D$
be a digraph.  Let $T$ and $B$ be pairwise disjoint subsets of $V(D)$ with $\vert T\vert=
\vert B\vert$.  Let the vertices in $T$ have out-degree $1$ and in-degree $0$ and those in
$B$ have in-degree $1$ and out-degree $0$.  Suppose that $D$ is embedded in $\Sigma$ such
that the vertices drawn on one boundary of the cylinder are exactly those in $T$ and those
drawn on the other are exactly those in $B$.  We say that $D$ is \dfn{$\{T,B\}$ cylinder-
embedded} in $\Sigma$.
  
\begin{thm}\label{cylinderreroute} Let $D$ be $\{T,B\}$ cylinder embedded in $\Sigma$.
Suppose that there are $\vert T\vert$ vertex-disjoint directed paths in $D$ from $T$ to $B$.
Suppose further that there is a family ${\cal F}$ of vertex-disjoint directed circuits that
each separate one of the boundaries of the cylinder from the other and which are consistently
oriented.  There is a butterfly minor of $D$ with these paths and circuits but with the
following additional property:\\

\noindent If $Q$ is any of the directed circuits and $P$ is a minimal sub-path of one
of the paths containing two vertices of $Q$, the union of $P$ and $Q$ contains a directed
circuit which does not separate the two boundaries. (Briefly, if the circuits go counter-
clockwise around the hole in the middle, the paths go clockwise.)
\end{thm}

\begin{proof} The proof is essentially the same as above.  The circuits in ${\cal F}$
assume the role of the horizontal paths in the previous proof and the $T$-$B$ paths assume
the role of the verticals.  Note that there is still a notion of top and bottom, which is
important to the previous proof.  We leave the conversion to the reader. (One possibility
is to cover the cylinder in the natural way with a bi-infinite strip, lift $D$ to this
cover, and note that having an infinite number of vertical paths does not impede the proof
in this case.)
\end{proof}


\section{No Cycles in the Proof}

In this section we present two lemmas for constructing \dfn{acyclic grids}.  An
acyclic grid of size $n$ is a digraph consisting of an ordered set of $n$ vertex-disjoint
``horizontal" directed paths and an ordered set of $n$ vertex-disjoint ``vertical" directed
paths.  We require that each vertical and horizontal path share exactly one vertex.  Moreover,
each vertical path hits the horizontal paths in order and similarly for the horizontal paths.

\begin{thm}\label{getagrid} Let $n$ be a positive integer.  There exists a number $v$
such that the following holds: Let $D$ be a digraph, $h=(3n-2)n$, and let $L$, $R$, $T$,
and $B$ be disjoint subsets of $V(G)$.  Let $\vert L\vert =\vert R\vert =h$ and $\vert T
\vert =\vert B\vert =v$.  Suppose that $D$ is $\{T,L,B,R\}$ disk embedded in $\Sigma$
with $h$ vertex-disjoint directed paths from $R$ to $L$ (horizontals) and $v$ vertex-disjoint
directed paths from $B$ to $T$ (verticals).  Suppose further that the horizontals and verticals
hit in reverse.  Then $D$ contains an acyclic grid of size $n$ that satisfies the following
properties:\\
\begin{enumerate}
\item{}Each horizontal path is contained below the top-most of the original horizontal
paths and begins in $R$.
\item{}Each vertical path is contained to the left of the right-most of the original
vertical paths and ends in $T$.
\end{enumerate}
\end{thm}

\begin{proof}
We label the horizontal paths from bottom to top as $H_1$, $H_2$, ..., $H_h$ and
label the vertical paths from left to right as $V_1$, $V_2$, ..., $V_v$.  For all $1\leq
i\leq v$ and $0\leq j\leq h$, let $V_i^j$ denote the sub-path of $V_i$ that:\\
\begin{enumerate}
\item{}begins at a vertex in $H_j$ (or $B$ in case $j=0$)
\item{}ends at a vertex in $H_{j+1}$ (or $T$ in case $j=h$)
\item{}is otherwise vertex-disjoint from the horizontal paths
\item{}occurs as late as possible along $V_i$
\end{enumerate}
Note that the paths $V_i^j$ for fixed $j$ occur from left to right in the disk between
$H_j$ and $H_{j+1}$ in the order $V_1^j$, $V_2^j$, ..., $V_v^j$.\\

We are interested in how the ends of the $V_i^{j-1}$ paths ``mix" with the starts of the
$V_i^j$ paths.  To this end, $\forall  1\leq i\leq v$ and $\forall  0\leq j\leq h$ let
$f_i^j$ denote the first vertex in $V_i^j$ and $l_i^j$ denote the last vertex in $V_i^j$.
Since the paths hit in reverse, $f_i^j$ occurs on $H_j$ before $l_i^{j-1}$.  For a subset
$S=\{i_1, ..., i_k\}$ of the vertical paths, $S$ is \dfn{integrated} at $j$ if the vertices
$f_{i_k}^j$, $l_{i_k}^{j-1}$, $f_{i_{k-1}}^j$, $l_{i_{k-1}}^{j-1}$, ..., $f_{i_1}^j$,
$l_{i_1}^{j-1}$ appear along $H_j$ in that order.  $S$ is \dfn{segregated} if the vertices
$f_{i_k}^j$, $f_{i_{k-1}}^j$, ..., $f_{i_1}^j$, $l_{i_k}^{j-1}$, $l_{i_{k-1}}^{j-1}$, ...,
$l_{i_1}^{j-1}$ appear along $H_j$ in that order.  For each $1\leq i\leq v$, we consider
the subpath of $H_j$ from $f_i^j$ to $l_i^{j-1}$.  It is well known that given many sub-paths
of a path, there are either many disjoint sub-paths or many sub-paths with common
intersection.  Combining this with the observation following the definition of $V_i^j$, we
see that:\\

\noindent{\it Claim:} For a fixed $j$ and any $m$, there is an integer $M$ such that a subset
$S$ of the $V_i$ of size at least $M$ contains a subset $S'$ of size at least $m$ that is
either integrated or segregated on $H_j$.\\

By iteratively applying this claim for all $1\leq j\leq h$, we see that we may choose
$v$ large enough such that there is a subset $S$ of the vertical paths of size $3n-2$ that is
either segregated or integrated on each of the horizontal paths.  We restrict the set of
vertical paths to $S$. (Paths outside $S$ are no longer considered ``vertical".) If we
remove some $H_{j_0}$ from the set of horizontal paths, we may relabel the set of
horizontal paths and then redefine the $V_i^j$.  The terminals of these paths remain the
same on all $H_j$ with $j\neq j_0$ except that the ends of the old $V_i^{j_0}$ on
$H_{j_0+1}$ (the $l_i^{j_0}$) are replaced by the ends of the new $V_i^{j_0-1}$ sub-paths
that go from $H_{j_0-1}$ to $H_{j_0+1}$ (the new $l_i^{j_0-1}$).  Since the paths hit in
reverse, for every $1\leq i\leq v$ the new $l_i^{j_0-1}$ occurs later on
$H_{j_0+1}$ than the old $l_i^{j_0}$.  So, if a subset of the vertical paths is
segregated on $H_{j_0+1}$, they remain so if $H_{j_0}$ is no longer considered a
horizontal path.\\

Since $h=n(3n-2)$, there are either $3n-2$ horizontal paths on which the
vertical paths are segregated or there are $n$ consecutive horizontal paths on which
the vertical paths are integrated.  We first suppose the former.  In this case, we
restrict the set of horizontal paths to some $3n-2$ on which the verticals are
segregated and restrict the verticals to an arbitrary subset of size $n$. (The
segregation remains by the previous paragraph.) We next select sub-paths of the
horizontal paths.  For a horizontal path $H_j$, we keep the start of $H_j$ (in $R$)
but change the end to the last (in the sense of $H_j$) $f_i^j$.  We refer to these
sub-paths as the horizontal paths and keep the order on these paths induced by the
original horizontal paths (lowest, highest, etc.) We claim that a vertical path $V_i$
may not hit some horizontal path $H_j$ and then hit a $H_{j'}$ with $j'<j$.  Otherwise,
there are vertices $u_1$ and $u_3$ on $V_i$ and $H_j$ and a vertex $u_2$ on $V_i$ and
$H_{j'}$ occuring on $V_i$ between $u_1$ and $u_3$.  Note that one of $u_1$, $u_3$ must
occur on $V_i$ before $f_i^{j-1}$.  But any such vertex must appear on the original $H_j$
(before we took sub-paths) after $l_i^{j_0}$ (since the paths hit in reverse).  So, any
such vertex is not included in the sub-path of $H_j$ that we have selected.\\

Therefore, each vertical path begins at a vertex in $B$, travels to the lowest horizontal
path, perhaps ``bubbles" backwards along this path (it does not bubble forward because
it hits the horizontal in reverse), and then proceeds to the next horizontal path, never
to return to the one it has just left.  Finally, the path ends at a vertex in $T$.  Since
we have not changed any of the vertical paths, property $2$ automatically follows.  To
complete the theorem in this case, we need to re-route the horizontal paths consistent
with property $1$.  To do so, we select the lowest horizontal path and travel from its
start to the first vertex we encounter along the right-most vertical path. We then travel
along this path until we get to the next highest horizontal path.  We then follow this
horizontal path to the next right-most vertical path, follow that vertical path to the next
horizontal path, and so on until we reach a vertex in the left-most vertical path.  The
path we have traveled will be the first of our final horizontal paths.  For the next one,
we start at the third lowest horizontal path (we start at the third instead of the
second to keep the paths vertex-disjoint) and repeat the procedure.  Since $V_i^j$ is below
the top-most horizontal for all $i$ and $j$ (except for those going from the top-most
horizontal to $T$), this procedure produces horizontal paths satisfying the theorem.\\

So, we may assume that there are $n$ consecutive horizontal paths on which the
$3n-2$ vertical paths are integrated.  We relabel these horizontal paths as $H_1$,...
$H_n$ (in ascending order) and relabel the vertical paths as $V_1$,...,$V_{3n-2}$ (left
to right).  For $1\leq i\leq 3n-2$ and $1\leq j\leq n$, we define $V_i^j$ as before.
(Note we are not defining $V_i^0$ for any $i$.) Defining $f_i^j$ and $l_i^j$ as before,
we see that for $2\leq j\leq n$ $l_{3n-2}^{j-1}$, $f_{3n-2}^j$, $l_{3n-3}^{j-1}$,
$f_{3n-3}^j$, ..., $l_1^{j-1}$, $f_1^j$ occur along $H_j$ in that order. (The only
interesting case is when $j=n$ since it is the one on the ``border" with non-integrated
original horizontal paths.) The horizontal paths of our grid will be the $n$ horizontal
paths we have selected.  For the right-most vertical path, we begin at $f_{3n-2}^1$ and
follow $V_{3n-2}^1$ to $l_{3n-2}^1$.  We then follow $H_2$ to $f_{3n-3}^2$.  Note that
we stay to the left of $V_{3n-2}$ as we travel along $H_2$.  We follow $V_{3n-3}^2$ from
$f_{3n-3}^2$ to $l_{3n-3}^2$ and then repeat the process until we follow $V_{2n-1}$ from
$f_{2n-1}^n$ to a vertex in $T$.  The path we have followed remains on the left of
$V_{3n-2}$ and so satisfies property $2$.  For the next vertical path we start at
$f_{3n-4}^1$ (again we ``skip" to make the new vertical paths vertex-disjoint) and proceed
as before.  The last vertical path begins at $f_n^1$ and ends by traveling along $V_1^n$
from $f_1^n$ to a vertex in $T$.
\end{proof}

Before presenting the next lemma, we need a definition.  A \dfn{bubble acyclic grid}
of size $n$ is an ordered set of $n$ vertex-disjoint ``horizontal" paths and an ordered
set of $n$ vertex-disjoint ``vertical" paths such that each vertical path hits all
horizontal paths and does so monotnoically with respect to the order on the
horizontals. (Note that we do not make such a restriction on the vertical paths.)
Moreover, we require that all paths hit in reverse.  So, a bubble acyclic grid is the
same as an acyclic grid except we allow the vertical paths to ``bubble" along a horizontal
before proceeding to the next.\\

\begin{thm}\label{getagrid2} Let $n$ be a positive integer.  There exists a number $h$
such that the following holds: Let $D$ be a digraph, $\Sigma$ a closed disk, and let $L$, $R$, $T$,
and $B$ be disjoint subsets of $V(G)$.  Let $\vert L\vert =\vert R\vert =h$
and $\vert T\vert =\vert B\vert =n$.  Suppose that $D$ is $\{T,L,B,R\}$ disk embedded
in $\Sigma$ with $h$ vertex-disjoint directed paths from $R$ to $L$ (horizontals) and $n$
vertex-disjoint directed paths from $B$ to $T$ (verticals).  Suppose further that the
horizontals and verticals hit in reverse.  Then $D$ contains a bubble acyclic grid of
size $n$ where each horizontal path is an original horizontal path.
\end{thm}

\begin{proof}
We first suppose that there are horizontal paths $H_1$, $H_2$ at distance $d$ (in
the top-bottom order) and a vertical path $P$ such that:\\

\noindent{\it Property {\cal P}:} There are vertices $h_1^1$, $h_2^1$, $h_1^2$, $h_2^2$,
..., $h_1^i$, $h_2^i$, ..., $h_1^{3n-2}$, $h_2^{3n-2}$ appearing on $P$ in that
order such that $h_1^i\in h_1$ and $h_2^i\in h_2$ $\forall 1\leq i\leq 3n-2$.\\

Since $P$ begins in $B$, we may assume that $H_1$ is below $H_2$.  For every
$1\leq i\leq 3n-2$, let $P_i$ denote the sub-path of $P$ beginning at $h_1^i$ and
ending at $h_2^i$ and set ${\cal V}=\{P_i:1\leq i\leq 3n-2\}$.  Let ${\cal H}$ be the
set of all horizontal paths between $H_1$ and $H_2$ inclusive.  ${\cal H}$ and
${\cal V}$ will be our new set of horizontal and vertical paths.  The new vertical
paths are sub-paths of an original vertical path that formed no directed digons
with the horizontal paths.  So, the vertical paths are integrated on each horizontal
path in the sense of the previous proof.  We may proceed to re-route the vertical
paths as in that proof to produce an acyclic grid of size $n$ with each horizontal
path an original horizontal path, proving the lemma in this case.\\

So, we may assume property {\cal P} fails for every vertical path.  Given a vertical
path $P$ and a horizontal path $Q$, we define $passes(Q,P)$ as the maximum number of distinct
vertices in $Q\cap P$ such that for any two of these vertices there is a vertex between
them on $P$ that is also on a horizontal path above $Q$. (So, $passes(Q,P)$ counts the
number of times $P$ hits $Q$ and then travels above $Q$.)  Since property ${\cal P}$
does not hold for any vertical path $P$, if we keep only every $d^{th}$ horizontal path
(ignoring the others) $passes(Q,P)<3n-2$ forall $Q$,$P$.  We will complete the proof
by proving the following claim by induction on $n^2-nk-l$.\\

\noindent{\it Claim:} There exists a positive integer $N(n,k,l)$ such that the following holds:
Suppose there are $n$ vertical paths, $N$ horizontal paths, and $passes(Q,P)<3n-2  \forall Q$,$P$.
In addition, suppose that $passes(Q,P)=1  \forall  P$ and each of the bottom $k$ horizontal
paths $Q$ and $passes(Q_0,P)=1$ for the left-most $l$ vertical paths $P$ and the $k+1$st
lowest horizontal path $Q_0$.  Then there is a bubble acyclic grid of size $n$ whose
horizontal paths are a subset of the original horizontal paths.\\

The claim is trivial if $n^2-nk-l\leq 0$.  Also, we may assume that $0\leq l<n$
by increasing $k$ (and setting $l=0$) in case $l=n$.  We will prove that $N(n,k,l)=(3n-3)
N(n,k,l+1)$ works.  Let $P$ be the $(l+1)^{st}$ vertical path and set $i=passes(Q_0,P)$.
Let $v_1$, ..., $v_i$ be a maximum set of vertices in $P\cap Q_0$ such that for any two
there is a vertex on $P$ between them and also on a horizontal path above $Q_0$.  We may
assume that these vertices appear along $P$ in that order.  For $1\leq j\leq i$, let
$height(j)$ denote the maximum height above $Q_0$ of a horizontal path which hits $P$
between $v_j$ and $v_{j+1}$.  Let $j_0$ be the first $j$ for which $height(j)\geq
N(n,k,l+1)j$.  We replace $P$ by the vertical path which has the same start as $P$ but
ends between $v_j$ and $v_{j+1}$ at the horizontal path of maximum height. (In case $j=i$
we keep the original $P$.) We also delete the $(j_{0}-1)N(n,k,l+1)$ horizontal paths
immediatly above $Q_0$.  With the new $P$, we have $passes(Q_0,P)=1$, and by the choice
of $j_0$ we may apply induction to complete the proof.
\end{proof}


\section{Getting the Goods}

In this section we begin the construction of the cylindrical grid.  We find the circuits
and the paths of the grid without too much regard to how they interact.  The following
sections will use the previous lemmas to straighten things out.

\begin{thm}\label{linkcuts} Let $n$ be a positive integer and $D$ an Eulerian digraph with
out-degree at most $6$.  Suppose $D$ has undirected vertex cuts (vertex cuts in the
underlying undirected graph) $C_1$ and $C_2$ such that the following hold:\\
\begin{enumerate}
\item{}$C_1$ and $C_2$ are disjoint, do not cross, and contain $7n$ vertices.
\item{}There is no undirected vertex cut of order less than $7n$ between $C_1$ and
$C_2$.
\end{enumerate}

Then, there are $n$ vertex-disjoint directed paths in $D$ from the vertices in $C_1$
to the vertices in $C_2$.
\end{thm}

\begin{proof}
If the conclusion of the theorem does not hold, there is a directed cut $C$ of
order less than $n$ separating $C_1$ from $C_2$ (edges may cross the cut but not from
a component containing a member of $C_1$ to one containing a member of $C_2$).  Since $D$
is eulerian of maximum out-degree $6$, there are at most $6n$ edges which cross the cut.
Since $C_1$ and $C_2$ are undirected cuts, for every edge that crosses $C$, there is an end
of that edge between $C_1$ and $C_2$ (inclusively).  Adding this end to $C$ for each edge
crossing $C$, we obtain an undirected cut of order less than $7n$ between $C_1$ and $C_2$,
contradicting (ii).
\end{proof}

\begin{thm}\label{findraw}Let $n$ be a positive integer.  There is a positive integer $N(n)$
such that if $D$ is a planar eulerian digraph of maximum out-degree $6$ with tree-width at
least $N$ there is a set of $n$ nested vertex-disjoint directed circuits oriented in the same
direction.  Moreover, there are $n$ vertex-disjoint directed paths from the inner circuit to
the outer and $n$ vertex-disjoint directed paths from the outer circuit to the inner. (The two
sets of paths need not be disjoint from one another).
\end{thm}

\begin{proof} In [], it is shown that an eulerian digraph of bounded out-degree has large
tree-width iff the underlying undirected graph has large tree-width.  This combined with
[Graph Minors 3 (?)] shows that for sufficiently large $N$ the undirected version of $D$
has an arbitrarily large undirected cylindrical grid as a minor. (An undirected cylindrical
grid is a cylindrical grid with no directions on the edges.) So we may choose $N$ so large
that $D$ has $120n$ circuits and $240n$ paths between the inner and outer circuit (without
taking minors).  Let the set of $120n$ circuits be $C$.  For any undirected vertex-cut $S$
in $D$ of order at most $7n-1$, there are at least $110n$ of the circuits in $C$ that contained
in the same component of $D-S$.  We call this component the \dfn{big side} of $S$.  Let
$s_i$ be a vertex on the inner-most circuit of $C$ and let $S_i$ be an undirected vertex cut
of order at most $7n$ with $s_i\notin S_i$ not in the big side of $S_i$, subject to that the
$s_i$ side of $S_i$ as large as possible, and subject to that $S_i$ as small as possible.
Let $s_o$ be a vertex on the outer-most circuit of $C$ and let $S_o$ be defined like $S_i$.
By construction $S_i$ and $S_o$ do not cross, $S_i$ is in the big side of $S_o$, and $S_o$
is in the big side of $S_i$.  Now, suppose $S$ is a vertex cut of order at most $7n$ separating
$S_i$ from $S_o$.  Then, $S$ lives in the big side of $S_i$ and $S_o$, and one of $S_i$,
$S_o$ is not in the big side of $S$.  We may assume that $S_i$ is not in the big side of $S$.
But then the big side of $S$ is a proper subset of the big side of $S_i$, and we have a
contradiction to the choice of $S_i$.  So, $S$ does not exist and $S_i$ and $S_o$ satisfy
the hypothesis of the previous theorem,.  Therefore, there are $n$ vertex-disjoint directed
paths from $S_i$ to $S_o$ and $n$ vertex-disjoint directed paths from $S_o$ to $S_i$.\\

One component of $D-(S_o\cup S_i)$ contains $100n$ of the circuits in $C$.  Let $H$ be $G$
restricted to this component but with the vertices in $S_i$ and $S_o$ added back. Let
$F_i$ be the face of $H$ containing $s_i$ in the superdigraph $D$ and similarly for $F_o$.
We may add at most $6n$ edges between the vertices in $S_i$ and at most $6n$ edges between
the vertices in $S_o$ to make $H$ Eulerian with maximum out-degree $6$ while remaining planar
(since $D$ is planar).  We let $F_i$ be any of the faces into which $F_i$ is split by adding
these edges and similarly for $F_o$.  If there are not $13n$ clockwise vertex-disjoint
directed circuits separating $F_i$ and $F_o$, there is a curve in the plane from inside $F_i$
to inside $F_o$ which intersects $H$ in fewer than $13n$ vertices and every edge of $H$ that
passes through the curve passes through in the counter-clockwise direction.  Since there are
at least $100n$ undirected circuits separating $F_i$ from $F_o$, we conclude that there are
at least $83n$ such edges.  We get a similar curve if there are not $13n$ counter-clockwise
circuits.  The union of these two curves can be partitioned into simple closed curves with
edges passing through them in only one direction.  There is one such closed curve with more
than six ($6*13<83$) times as many edges crossing it as the number of vertices on it.  As this is
impossible in an Eulerian digraph with maximum out-degree $6$, we may assume there are $13n$
clockwise circuits in $H$ separating $F_i$ from $F_o$.  So, there are at least $13n-2*6n$
clockwise circuits in $D$ separating $S_i$ from $S_o$.  Together with the $n$ paths from
$S_o$ to $S_i$ and the $n$ paths from $S_i$ to $S_o$ we have the desired structure.
\end{proof}


\section{Putting It All Together}

In this section we ``clean up" the paths and circuits we found in the previous section
to form the cylindrical grid.

\begin{thm}\label{getcylgrid} Let $n$ be a positive integer.  There is a positive
integer $N$ such that the following holds: Let $D$ be a digraph embedded in the plane
that consists of the union of $N$ nested counterclockwise directed circuits $C_1$,...,$C_N$
(nested in that order), $N$ vertex-disjoint directed paths contained between $C_1$ and
$C_N$ beginning on $C_N$ and ending on $C_1$, and $N$ vertex-disjoint directed paths
contained between $C_1$ and $C_N$ beginning on $C_1$ and ending on $C_N$.  $D$ has a
cylindrical grid butterfly-minor of size $n$.
\end{thm}

\begin{proof}
We call the paths that begin on $C_N$ the \dfn{in-paths} and those that begin on $C_1$
the \dfn{out-paths}.  We may assume that all paths only intersect $C_1$ and $C_N$ at
their terminals.  We begin by applying \rft{cylinderreroute}.  For each in-path, let
the \dfn{outer part} be the sub-path that has the same start as the path intersects
$C_{N/2}$ only at its end.  For each out-path, let the \dfn{outer part} be the sub-path
that has the same end as the path and intersect $C_{N/2}$ only at its start.  By
bipartite Ramsey, either there are many in-paths and many out-paths with vertex-disjoint
outer parts or many in-paths and many out-paths with pairwise intersecting outer parts.
Since we are free to choose $N$ as large as necessary, we may assume that the original
set of in-paths and out-paths either have disjoint or intersecting outer parts as
described.  We consider the two cases separately.\\

Suppose first that the in-paths and out-paths have disjoint outer parts.  We ignore the
circuits $C_1$, ..., $C_{N/2-1}$ and restrict the in and out paths to their outer parts.
We further restrict attention to an arbitrary $3n-1$ of the circuits.  For any circuit
$C_k N/2<k<N$ and an in-path $p$, $p$ may not contain a sub-path $p'$ that begins and
ends on $C_k$ such that the union of $p'$ and the sub-path of $C_k$ that begins at the
end of $p'$ and ends at the beginning of $p'$ separates $C_{N/2}$ from $C_N$. (In other
words, $p$ may not ``wrap around" $C_k$.) This can be seen by considering any out-path,
using the fact that the in-paths and out-paths are disjoint, and recalling the property
supplied by \rft{cylinderreroute}.  Similarly, the out-paths do not wrap around $C_k$.\\

For each circuit $C_k N/2<k<N$ and an in- or out-path $p$, we define the \dfn{interval} of
$p$ on $C_k$ as the directed sub-path of $C_k$ that begins at the last vertex $p$
shares with $C_k$ and ends at the first vertex $p$ shares with $C_k$.  Since the paths
do not wrap around $C_k$, all vertices $p$ shares with $C_k$ occur in the interval.
Again by applying disjointness and the property of \rft{cylinderreroute}, the
intervals of any in-path and any out-path must be disjoint.  Furthermore, by removing
some of the in- and out-paths, we may assume that each $C_k$ can be partitioned into
two sub-paths $I_k$ and $O_k$ such that all the in-path intervals are in $I_k$ and
all the out-path intervals are in $O_k$.\\

We apply \rft{getagrid2} twice, first with the set of in-paths as the verticals and the
set of all $I_k$ as the horizontals and then with the set of out-paths as the verticals
and the set of all $O_k$ as the horizontals.  (Due to the relative orientation of the
horizontals and verticals, \rft{getagrid2} is only applicable in the case of the out-paths,
but the necessary lemma for the in-paths holds with the same proof.) If \rft{getagrid2}
produced acyclic grids instead of bubble acyclic grids, the result would be a cylindrical
grid of size at least $n$ for sufficiently large $N$.\\

In general, to get an acyclic grid from a bubble acyclic grid, we re-route the horizontal
paths in the typical fashion.  We start at the beginning of the lowest horizontal path,
follow it to the first intersection with a vertical path, follow that vertical path to the
next horizontal path, and continue until we have traveled along all vertical paths.  We
then start on the third lowest horizontal path and repeat the procedure.\\

In our particular case, before we re-route we restrict our attention to $n$ of the
vertical paths in the bubble acyclic grid produced using the in-paths and similarly
in the bubble acyclic grid produced using the out-paths.  We then produce acyclic
grids of size $n$ by performing the above procedure. ($3n-1$ is the number of
horizontals needed to be able to perform the re-routing procedure $n$ times.)  Since
the horizontal paths are actually sub-paths of circuits, re-routing them is dangerous
because the circuits may no longer ``close up." However, we shift in when re-routing
the horizontals in the in-path grid and shift out by the same amount when re-routing
in the out-path grid.  So, we succeed in producing a cylindrical grid of size $n$ in
this case.\\

We next suppose that the in-paths and out-paths have intersecting outer parts.  By
removing some in- and out-paths, we may assume that the in-paths start in a sub-path
of $C_N$ that does not contain the end of any out-path.  We may then speak of the
``middle" in-path and out-path.  The outer parts of these two paths intersect, and
we define the path $p$ by starting at the middle in-path, traveling until we encounter
a vertex on the outer part of the middle out-path, and then following the
middle out-path to its end.  Since $p$ is a subset of outer parts, it does not intersect
$C_{N/2}$.  In particular, $C_1$ is contained on one side of $p$ (viewing $C_N$ as
bounding a disk and $p$ as separating that disk) and we assume without loss of generality
in our future argument that the directed circuit formed by $p$ and a sub-path of $C_N$
does not contain $C_1$.\\

We construct a new digraph $H$ embedded in a disk.  The boundary of the disk will be
$C_N$.  Let the sub-path of $C_N$ that forms a directed circuit with $p$ be $q$.
Half of the in-paths and half of the out-paths have terminals on $q$, and we restrict
our attention to these (including the middle in-path and out-path).  We begin the
construction by retaining all sub-paths of in- and out-paths that are contained inside
the directed circuit formed by $p$ and $q$. (An in- or out-path may have many such sub-paths.)
In $D$, an in- or out-path may have a sub-path that starts on $p$, goes around
$C_1$ and returns to $p$.  We call such a sub-path and the path that contains it
\dfn{deviant}.  We have four cases based on whether at least half the in-paths are not deviants
and whether at least half the out-paths are not deviants.  All cases are handled nearly
identically, and we start with the case that at least half the out-paths are not deviants
and half the in-paths are not deviants.  By restriction, we may assume that all in-paths
and out-paths are not deviants.\\

We currently only have a set of disjoint sub-paths in $H$ for each of the in- and
out-paths under consideration.  However, since there are no deviants, we may add an edge
to $H$ representing each missing sub-path that starts and ends on $p$ (having the same
start and end as the sub-path) and embed these edges in the disk on the same side of $p$
as $C_1$. (No deviants means we may topologically ``pull apart" the sub-paths of the in-paths
from those of the out-paths.)  For the sub-paths that start (or end) on $p$ but have
their other terminal on $C_1$, we replace them by an edge starting (or ending) at the same
place on $p$ and terminating on a new vertex on the boundary of the disk.  This can be done
in such a way that the starts of the in-paths, the starts of the out-paths, the ends of
the in-paths, and the ends of the out-paths occur in that counter-clockwise order around
the boundary of the disk.  We call the edges we've added to represent sub-paths
\dfn{phantom} edges.\\

We next apply \rft{getagrid} with the in-paths as the verticals and the out-paths
as the horizontals.  By properties $1$ and $2$ of that lemma, each vertical path ends in a
phantom edge ending on $C_N$ and each horizontal path begins with a phantom edge beginning on
$C_N$.  Moreover, once a vertical path uses a phantom edge, it hits no more horizontals and
once a horizontal hits a vertical, it uses no more phantom edges.  The presence of such a grid
implies the existence of an acyclic grid minor in $D$ that is completely contained
outside $C_{N/2}$ except for the starts of the horizontal paths and the ends of the
vertical paths which start and end respectively on $C_1$.  Using the circuits $C_1$, ...,
$C_{N/2}$ it is then easy to find an acyclid grid and a set of vertex-disjoint directed paths
(from each other and the grid) linking the ends of the verticals to the starts of the
horizontals.  We may make this grid arbitrarily big by choosing a large enough $N$, and a
sufficiently big \dfn{linked acyclic grid} of this kind contains the cylindrical grid of size
$n$.\\

We next consider the case when at least half the in-paths and half the out-paths are
deviant.  As usual we may then assume that all in- and out-paths are deviant.  We
proceed as before and add edges representing the missing sub-paths, but only do so
for sub-paths of the in-paths before the first deviant sub-path and for sub-paths of the
out-paths after the last deviant sub-path.  With this restriction, we may add the phantom
edges as before to maintain the embedding of $H$ in the disk.  For the first deviant
sub-path on an in-path we add a new edge with the same start as the deviant sub-path
and with end on the boundary of the disk.  For the last deviant sub-path on an out-path
we add a new edge with the same end as the deviant sub-path and starting on the boundary
of the disk.  Since we are replacing the first and last deviant sub-paths, we may add
these edges in such a way that the terminals of the in- and out-paths we construct appear
around the boundary of the disk in the same order as in the previous case.  Applying
\rft{getagrid} and translating the result back to $G$ gives an acyclic grid with
vertical paths that end with deviant sub-paths and horizontal paths that begin with
deviant sub-paths.  Since every deviant horizontal sub-path hits every deviant vertical
sub-path, we may once again find a linked acyclic grid and complete the proof.\\

The two mixed cases are handled by combining the arguments for the first two cases.
Deviant paths are truncated and then represented by phantom edges while non-deviant
ones are represented by these edges without truncation.  We may find the linked acyclic
grid in $D$ since every deviant sub-path hits every sub-path with terminals on $p$ and
$C_1$.
\end{proof}

\end{document}